\documentclass[AMA,STIX1COL]{WileyNJD-v2}

\articletype{Special Issue Paper}%

\received{5 June 2020}
\revised{x xxxx 2020}
\accepted{x xxxx 2020}

\usepackage{amsmath}
\usepackage{amssymb}
\usepackage{amsthm}
\usepackage{colortbl}
\usepackage{enumerate}
\usepackage{tikz-cd}

\newcolumntype{x}[1]{
>{\centering\arraybackslash}p{#1}}%

\newcommand{\RN}{\mathbb{R}^{n+1}}
\newcommand{\R}{\mathbb{R}}
\newcommand{\Rn}{\mathbb{R}^{n}}
\newcommand{\Ru}{\mathbb{R}^{n+1}_+}
\newcommand{\Cln}{\mathcal{C}l_{n-1}}
\newcommand{\ClN}{\mathcal{C}l_n}
\newcommand{\C}{\mathbb{C}}

\newcommand{\SN}{\mathbb{S}^{n}}
\newcommand{\RP}{\mathbb{R}P^{n}}

\newcommand{\be}{\begin{align*}}
\newcommand{\ee}{\end{align*}}

\newcommand{\wzwo} {\stackrel{\circ}{{W}^1_2}}
\newcommand{\wzwop} {\stackrel{\circ}{{W}^k_p}}

\numberwithin{equation}{section}

\begin{document}

\title{The $\Pi$-operator on Some Conformally Flat Manifolds and Hyperbolic Half Space}

\author[1,2]{Wanqing Cheng*}

\author[1]{John Ryan}

\authormark{W. Cheng, J. Ryan}

\address[1]{\orgdiv{Department of Mathematical Science}, \orgname{University of Arkansas, Fayetteville}, \orgaddress{\state{AR}, \country{U.S.A.}}}

\address[2]{\orgdiv{Department of Physics}, \orgname{University of Arkansas, Fayetteville}, \orgaddress{\state{AR}, \country{U.S.A.}}}
%

\corres{*\email{wcheng@uark.edu}}

\presentaddress{University of Arkansas, Fayetteville, \\ Department of Physics,\\ Fayetteville, AR, 72701, U.S.A.}

\abstract[Abstract]{The $\Pi$-operator, also known as Ahlfors-Beurling transform, plays an important role in solving the existence of locally quasiconformal solutions of Beltrami equations. In this paper, we first construct the $\Pi$-operator on a general Clifford-Hilbert module. This $\Pi$-operator is also an $L^2$ isometry. Further, it can also be used for solving certain Beltrami equations when the Hilbert space is the $L^2$ space of a measure space. Then, we show that this technique can be applied to construct the classical $\Pi$-operator in the complex plane and some other examples on some conformally flat manifolds, which are constructed by $U/\Gamma$, where $U$ is a simply connected subdomain of either $\RN$ or $\SN$, and $\Gamma$ is a Kleinian group acting discontinuously on $U$. The $\Pi$-operators on those manifolds also preserve the isometry property in certain $L^2$ spaces, and their $L^p$ norms are bounded by the $L^p$ norms of the $\Pi$-operators on $\RN$ or $\SN$, depending on where $U$ lies. The applications of the $\Pi$-operator to solutions of the Beltrami equations on those
 conformally flat manifolds are also discussed. At the end, we investigate the $\Pi$-operator theory in the upper-half space with the hyperbolic metric. }

\keywords{$\Pi$-operator, Beltrami equations, Conformally flat manifolds, Spectrum.}

\jnlcitation{\cname{%
\author{W. Cheng},
and \author{J. Ryan}} (\cyear{2020}),
\ctitle{The $\Pi$-operator on Some Conformally Flat Manifolds and Hyperbolic Half Space}, \cjournal{Math Meth Appl Sci}, \cvol{2020;xx:x--x}.}

\maketitle

\section{Introduction}
The $\Pi$-operator is one of the tools used to study smoothness of functions over Sobolev spaces and to solve the Beltrami equations. In one dimensional complex analysis, the Beltrami equation is given by
$\displaystyle \frac{\partial w}{\partial \overline{z}}=\mu \displaystyle \frac{\partial w}{\partial z}$,
where $\mu=\mu(z)$ is a given complex function, and $z\in\mathbb{C}$. It can be transformed to a fixed-point equation $h=q(z)(I+\Pi_\Omega h),$
where
$$\Pi_\Omega h(z)=-\displaystyle\frac{1}{\pi i}\displaystyle\int_\Omega \frac{h(\xi)}{(\xi-z)^2}d\xi_1 d\xi_2$$
is the complex $\Pi$-operator. This singular integral operator acts as an isometry from $L^2(\mathbb{C})$ to $L^2(\mathbb{C})$ with the $L^p$-norm being a long standing conjecture by Iwaniec.
\par
With the help of Clifford algebras, the classical Beltrami equation and $\Pi$-operator with some well known results can be generalized to higher dimensions. Abundant results in Euclidean space have been found. For instance, in \cite{GKS}, G\"{u}rlebeck, K\"{a}hler and Shapiro considered a class of generalizations of the complex one-dimensional $\Pi$-operator in spaces of quaternion-valued functions depending on four real variables. In \cite{GK}, G\"{u}rlebeck and K\"{a}hler provided a hypercomplex generalization of the complex $\Pi$-operator which turns out to have most of the properties of its origin in one dimensional complex analysis. K\"{a}hler studied Beltrami equations in the case of quaternions in \cite{Kahler}, which gave an overview of possible generalizations of complex Beltrami equation and their properties in the quaternionic case. In \cite{Blaya}, the authors studied the $\Pi$-operator in Clifford analysis by using two orthogonal bases of a Euclidean space, which allows to find the expression of the jump of the generalized $\Pi$-operator across the boundary of the domain. The case of the $\Pi$-operator and the Beltrami equation on the unit sphere has also been discussed in \cite{CRK} with most useful properties inherited from the complex $\Pi$-operator. The classical Ahlfors-Beurling inequality has also been generalized to higher dimensions by Martin in \cite{Martin}.
\par
Conformally flat manifolds are manifolds with atlases whose transition maps are M\"{o}bius transformations. They can be parametrized by $U/\Gamma$ where $U$ is a simply connected subdomain of either $\SN$ or $\Rn$ and $\Gamma$ is a Kleinian group acting discontinuously on $U$. Examples of such manifolds treated here include the real projective space $\mathbb{R}P^n$, cylinders and Hopf manifolds $\mathbb{S}^1\times \SN$. More details for these conformally flat manifolds can be found in \cite{KR,KR1}. In the present paper, we will generalize the results in Euclidean space \cite{GK} and on the unit sphere \cite{CRK} to the previous conformally flat manifolds through proper projection maps.
\par
This paper is organized as follows. In Section 2, we briefly introduce the Clifford algebras setting and some integral formulas. Section 3 is devoted to an introduction of $\Pi$-operator in a general Hilbert space. It turns out that this technique can be applied to obtain the results of the $\Pi$-operator in the classical case in the complex plane, and the $\Pi$-operators on some other conformally flat manifolds can also be constructed with this strategy. This is explained in details in the rest of the paper. More specifically, in section 4, we define the real projective space $\RP$ as a quotient space of the $n$-dimensional unit sphere with certain projection map. With the help of this projection map we can induce the Dirac operator, Cauchy transform, some integral formulas and the $\Pi$-operator from $\SN$ to $\RP$. The Beltrami equation on the real projective space is also studied here as an application. In Section 5, we generalize the results in Euclidean space to cylinders and Hopf manifolds. Applications to the Beltrami equations on cylinder and Hopf manifolds are also provided. Section $6$ is devoted to an investigation for the $\Pi$-operator theory on the upper-half space with the hyperbolic metric. Such defined $\Pi$-operator also possesses most properties that it has in one dimensional complex analysis.
\subsection*{Acknowledgements}
This paper is dedicated to Klaus G\"urlebeck on his 65th birthday.
\section{Preliminaries}
\subsection{Clifford analysis in Euclidean space}
Let $\{e_1,\cdots,e_n\}$ be the canonical orthonormal basis of the Euclidean space $\Rn$. The real Clifford algebra $\ClN$ is generated from $\Rn$ by considering the relationship $e_ie_j+e_je_i=-2\delta_{ij}e_0$, where $e_0$ is the identity of $\ClN$ and $\delta_{ij}$ is the usual Kronecker symbol. An arbitrary element of the basis of the Clifford algebra can be written as ${e}_A={e}_{j_1}\cdots {e}_{j_r},$ where $A=\{j_1, \cdots, j_r\}\subset \{1, 2, \cdots, n\}$ and $1\leq j_1< j_2 < \cdots < j_r \leq n$. Hence, for any $a\in \mathcal{C}l_n$, we have $a=\sum_Aa_Ae_A,$ where $a_A\in \mathbb{R}$. The norm of a Clifford number $x$ is defined as
 $\|x\|^2=\sum_{A\subset\{1,\cdots,n\}}x_A^2.$
 If the set $A$ contains $k$ elements, then we call $e_A$ a \emph{k-vector}. Likewise, we call each linear combination of $k$-vectors a $k$-vector. The vector space of all $k$-vectors is denoted by $\Lambda^k\Rn$. Obviously, $\ClN$ is the direct sum of all $\Lambda^k\Rn$ for $k\leq n$. In particular, under the rule of multiplication, each non-zero vector $x\in\Rn$ has a multiplicative inverse $x^{-1}=\frac{-x}{||x||^2}$. see \cite{Br} for more details on Clifford algebras. We also need the following three anti-involutions in Clifford analysis.
\begin{itemize}
\item \textbf{Reversion:}
$
\tilde{a}=\sum_{A} (-1)^{|A|(|A|-1)/2}a_Ae_A,
$
where $|A|$ is the cardinality of $A$. In particular, $\widetilde{e_{j_1}\cdots e_{j_r}}=e_{j_r}\cdots e_{j_1}$. 
\item \textbf{Clifford conjugation:}
$
a^{\dagger}=\sum_{A} (-1)^{|A|(|A|+1)/2}a_Ae_A,
$
satisfying ${e_{j_1}\cdots e_{j_r}}^{\dagger}=(-1)^re_{j_r}\cdots e_{j_1}$. 
\item \textbf{Clifford involution:}
$
\bar{a}=\tilde{a}^{\dagger}=\widetilde{a^{\dagger}}.
$
\end{itemize}
In the rest of this paper, we identify the Euclidean space $\RN$ with the direct sum $\Lambda^0\Rn\oplus\Lambda^1\Rn$ and $\Omega\subset \RN$ is a domain with a sufficiently smooth boundary $\Gamma=\partial \Omega$. Further, we only deal with functions defined in $\Omega$ taking values in $\ClN$. These functions can be written as
\begin{eqnarray*}
f(x)=\sum_{A\subseteq\{ e_1,e_2,...e_n \}}f_A(x)e_A,\quad x\in \Omega.
\end{eqnarray*}
Properties such as continuity, differentiability, integrability, and so on, which are ascribed to $f$ have to be possessed by all components $f_A(x),\ A\subseteq\{ e_1,e_2,...e_n \}$. The spaces $C^k(\Omega,{\mathcal{C}l_n})$ and $L^p(\Omega, {\mathcal{C}l_n}) $ are defined as right Banach modules with the corresponding traditional norms.  In particular, the space $L^2(\Omega,{\mathcal{C}l_n})$ is a right Hilbert module equipped with a ${\mathcal{C}l_n}$-valued sequilinear form
$$
\langle u,v\rangle=\int_\Omega \overline{u(\eta)} v(\eta)\, d\eta.
$$
Furthermore, $W_p^k(\Omega,{\mathcal{C}l_n}), k\in \mathbb{N}\cup\{0\},1\leq p<\infty$ denotes the Sobolev space as the right module of all functionals whose derivatives belong to $L^p(\Omega,{\mathcal{C}l_n})$, with norm
  $$
  \|f\|_{W_p^k(\Omega,{\mathcal{C}l_n})}:=\big(\sum_{A}\sum_{\|\alpha\|\leq k}\|D^\alpha_w f_A\|_{L^p(\Omega,{\mathcal{C}l_n})}^p\big)^{1/p}.
  $$
The closure of the space of test functions $C^\infty_0(\Omega, {\mathcal{C}l_n})$ in the $W_p^k$-norm will be denoted by $\wzwop(\Omega, {\mathcal{C}l_n})$.
\par
The Euclidean Dirac operators $D_x$ and $D_0$ arise as generalizations of the Cauchy-Riemann operator in one dimensional complex analysis and
$
D_x:=\sum_{i=1}^{n}e_i\partial_{x_i},\ D_0:=e_0\partial_{x_0}+\sum_{i=1}^{n}e_i\partial_{x_i}=e_0\partial_{x_0}+D_x.
$
Note $D_x^2=-\Delta_n$, where $\Delta_n$ is the Laplacian in $\Rn$, and $\Delta_{n+1}=D_0\overline{D_0}$. A $\ClN$-valued function $f(x)$ defined on a domain $\Omega$ in $\RN$ is called left monogenic if $D_0f(x)=\sum_{i=0}^{n}e_i\partial_{x_i}f(x)=0.$ Since Clifford multiplication is not commutative in general, there is a similar definition for right monogenic functions.
\par
Let $f \in C^1(\Omega, {\mathcal{C}l_n})$, $G(x,y)=\displaystyle\frac{\overline{x-y}}{\|x-y\|^{n+1}}$ is the fundamental solution of $D_0$ (see \cite{LR}). When considering functions with compact support, $D_0$ has a left and right inverse (called Cauchy transform) $T_{\Omega}$ as follows.
\begin{eqnarray*}
T_\Omega f(x)=\frac{1}{\omega_n}\int_\Omega G(x,y)f(y)dy,
\end{eqnarray*}
where $\omega_{n+1}$  is the area of the $n$-dimensional unit sphere. For more details, see \cite{NW}. Also, there is a non-singular boundary integral operator given by
\begin{eqnarray*}
F_{\partial \Omega}f(x)=\frac{1}{\omega_n}\int_{\partial \Omega}G(x,y)n(y)f(y)d\sigma(y).
\end{eqnarray*}
With the above two integral operators, we have the classical Borel-Pompeiu formula in Clifford analysis as follows.
\begin{theorem} \cite{GK} For $f\in C^1(\Omega,\ClN)\cap C(\overline\Omega)$, we have
\begin{eqnarray*}
f(x)=\frac{1}{\omega_n}\int_{\partial \Omega}G(x,y)n(y)f(y)d\sigma(y)+\frac{1}{\omega_n}\int_\Omega G(x,y)D_0f(y)dy,
\end{eqnarray*}
In particular, if $f\in \wzwo(\Omega,{\mathcal{C}l_n})$, then
\begin{eqnarray*}
f(x)=\frac{1}{\omega_n}\int_\Omega G(x,y)D_0f(y)dy.
\end{eqnarray*}
\end{theorem}
\subsection{Clifford analysis on the unit sphere}
Recall that the generalized spherical Dirac operator $D_s$ and its conjugate  on the $n$-dimensional unit sphere $\SN$ are defined as follows:
$D_s=x(\Gamma_0-\frac{n}{2}),\
\overline{D_s}=\overline{x}(\overline{\Gamma_0}-\frac{n}{2}),$
where $\Gamma_0=\sum_{j=1}^{n} e_0e_jL_{0,j}-\sum_{i=1,j>i}^{n} e_ie_jL_{i,j}$, and here the operators $L_{i,j}=x_i\partial_{x_j}-x_j\partial_{x_i}$ are called the angular momentum operators. It is well known that the fundamental solution of $\Gamma_0$ is $G_s(x,y)=\frac{\overline{x-y}}{\|x-y\|^n}$,
and the fundamental solution of $\overline{\Gamma_0}$ is $\overline{G_s(x,y)}=\frac{x-y}{\|x-y\|^n}$, $x,y\in \SN$, see \cite{CRK} for details.
\par
Assume $\Omega$ is a bounded smooth domain on ${\SN}$ and $f \in C^1(\Omega, \ClN)$. One can define Cauchy transforms with respect to $D_s$ and $\overline{D_s}$ as below \cite{CRK}.
\begin{eqnarray*}
T_\Omega f(x)=\frac{1}{\omega_n}\int_\Omega G_s(x,y)f(y)dy=\int_\Omega \frac{\overline{x-y}}{\|x-y\|^n}f(y)dy,\quad
\overline{T}_\Omega f(x)=\frac{1}{\omega_n}\int_\Omega \overline{G_s(x,y)}f(y)dy=\int_\Omega \frac{x-y}{\|x-y\|^n}f(y)dy.
\end{eqnarray*}
Here, $T_{\Omega}$ ($\overline{T}_{\Omega}$) is also a left and right inverse for $D_s$ ($\overline{D_s}$) when considering functions with compact support, see Theorem \ref{BPF} below.
Also, we have two non-singular boundary integral operators
\begin{eqnarray*}
F_{\partial \Omega}f(x)=\frac{1}{\omega_n}\int_{\partial \Omega}G_s(x,y)n(y)f(y)d\sigma(y),\quad
\overline{F}_{\partial \Omega}f(x)=\frac{1}{\omega_n}\int_{\partial \Omega}\overline{G_s(x,y)}n(y)f(y)d\sigma(y).
\end{eqnarray*}
Then the Borel-Pompeiu formula for $D_s$ and $\overline{D_s}$ is stated as follows.
\begin{theorem}[Borel-Pompeiu formula\cite{LR}]\label{BPF}
\hfill\\
For $f \in C^1(\Omega)\cap C(\overline\Omega)$, we have
\begin{eqnarray*}
f(x)=\frac{1}{\omega_n}\int_{\partial \Omega}G_s(x,y)n(v)f(v)d\sigma(y)+\frac{1}{\omega_n}\int_\Omega G_s(x,y)D_sf(y)dy, \label{1}
\end{eqnarray*}
in other words, $f=F_{\partial \Omega}f+T_\Omega D_sf$. Similarly, $f=\overline{F}_{\partial \Omega}f+\overline{T}_\Omega \overline{D_s}f$
\begin{eqnarray*}
f(x)=\frac{1}{\omega_n}\int_{\partial \Omega}\overline{G_s(x,y)}n(y)f(y)d\sigma(y)+\frac{1}{\omega_n}\int_\Omega \overline{G_s(x,y)}\overline{D_s}f(y)dy, \label{1}
\end{eqnarray*}
In particular, if $f$ has compact support, then $T_{\Omega}D_s=\overline{T_{\Omega}}\overline{D_s}=I$.
\end{theorem}
\section{The $\Pi$-operator in Hilbert space}
In this section, we will provide a $\Pi$-operator defined on a general Hilbert space. This $\Pi$-operator has the isometry property, which motivates the definitions of $\Pi$-operators in different conformally flat manifolds in the following sections.
\par
Let $H$ be a real Hilbert space, $\mathcal{S}$ is a dense subspace of $H$. Let $f,g\in \mathcal{S}\otimes \ClN$,
and $D$ is a linear map from $\mathcal{S}\otimes \ClN$ to itself. Further, $D$ also satisfies $D^*D=DD^*$ where $D^*$ is the dual operator of $D$ in the sense of
$
\langle Df,g\rangle=\langle f,D^*g\rangle,
$
where $\langle\ ,\ \rangle$ is the inner product on $H$. Suppose $G$ is an operator acting on $\mathcal{S}\otimes \ClN$, then it is called the inverse of $D$ if it satisfies $DG=GD=I$.
\begin{definition}
The generalized $\Pi$-operator in the Hilbert space $H$ is defined as
$$\Pi=D^*G.$$
\end{definition}
Next, we will show that the generalized $\Pi$-operator defined above also has the isometry property. 
\begin{theorem}
The generalized operator $\Pi=D^*G$ is an isometric operator in $H\otimes \ClN$.
\end{theorem}
\begin{proof}
\begin{eqnarray*}
\langle \Pi f,\Pi g\rangle=\langle D^*Gf,D^*Gg\rangle=\langle Gf,DD^*Gg\rangle 
=\langle Gf,D^*DGg\rangle=\langle DGf,DGg\rangle=\langle f,g\rangle.
\end{eqnarray*}
\end{proof}
Further, our generalized $\Pi$-operator can also be used to solve certain Beltrami equations. More specifically, if we let $H$ be $L^2(X)$, where $X$ is a measure space with a measure $\eta$, then we can define a Beltrami equation over $H\otimes\ClN$ i.e., $L^2(X,\ClN)$ as 
$Df=qD^*f,$
where $q\in L^{\infty}(X,\ClN) $. This is similar as in Euclidean space with the essential supremum norm with respect to $\eta$. By the
substitution $f=\phi+Gh$ where $\phi$ is a solution for $D\phi=0$, we transform the Beltrami equation in the following way.
\begin{eqnarray*}
D(\phi+Gh)=h=qD^*(\phi+Gh)=q(D^*\phi+\Pi h).
\end{eqnarray*}
Hence, if $h$ is the unique solution of the equation
$h=q(D^*\phi+\Pi h),$
then $f=\phi+Gh$ is the unique solution of the Beltrami equation. The Banach fixed point theorem tells us this equation has a unique solution if $\|q\|\leq q_0< \frac{1}{\|\Pi\|}$, with $q_0$ being a constant. Hence, as in the classical case, the problems of the existence of the solution to the Beltrami equation becomes the norm estimate of the generalized $\Pi$-operator.
\par
As special cases of this general Hilbert space approach, one has the $L^2$ isometry of the usual $\Pi$-operator in one complex variable and the $\Pi$-operator in $\mathbb{R}^n$ described in \cite{Blaya,GK,GKS} and elsewhere. The next sections describe the $\Pi$-operator acting over $L^2$ spaces over other manifolds.
\section{$\Pi$-operators on real projective space}
Recall the construction of our $\Pi$-operator in the previous section, if we let $X$ be the real projective space $\RP$ with the measure $\eta$ by pushing forward the Lebesgue measure on $\SN$. Then, $H=L^2(\RP,\mathbb{R})$ becomes a real Hilbert space,  and $H\otimes \ClN$ is a Clifford-Hilbert module with the inner product
\begin{eqnarray*}
\langle f,g\rangle=\int_{V'}\overline{f}(x)g(x)d\eta(x),
\end{eqnarray*}
where $V'$ is a subset of real projective space with $\overline{V'}$ inclosed and $f, g: V'\longrightarrow \ClN$. Therefore, we can obtain the $\Pi$-operator theory on real projective space as a special case of Section $3$. More details are given below.
\subsection{Dirac operators on real projective space}
We know that the real projective space $\RP$ is defined as $\SN/\Gamma$, where $\Gamma=\{\pm1\}$. This implies that $\Pi$-operator theory on the real projective space can be generalized from the $\Pi$-operator theory on the unit sphere. Notice that there is a projection map $p: \SN\longrightarrow\RP$, such that for each $x\in \SN$, $p(\pm x)=x'$. If Q is a subset of $\SN$, we denote $p(\pm Q)=Q'$. First, we consider the bundle $E_1$ by making the identification of $(x,X)$ and $(-x, X)$ where $x\in \SN$ and $X\in \ClN$.
\par
Now we change the spherical Cauchy kernel $G_s(x,y)=- \frac{1}{\omega_n}\frac{\overline{x-y}}{\|x-y\|^n}$, $x,y\in \SN$ into a kernel which is invariant with respect to $\Gamma=\{\pm1\}$, and this gives us a kernel $G_{\RP_1}(x,y)=G_s(x,y)+G_s(-x,y)$ for $\RP$ \cite{KR}.
\par
Suppose $S$ is a suitably smooth hypersurface lying in the northern hemisphere of $\SN$ and $V$ is also a domain lying in the northern hemisphere sphere that $S$ bounds a subdomain $W$ of $V$. If $f: V \longrightarrow \ClN$ is a left spherical monogenic function and $y\in W$, then we have
\begin{eqnarray*}
f(x)=\frac{1}{\omega_n}\int_S \big(G_s(x,y)+G_s(-x,y)\big)n(y)f(y)d\sigma(y),
\end{eqnarray*}
where $n(y)$ is the unit outer normal vector to $S$ at $x$ lying in the tangent space of $\SN$ at $y$. Now we use the projection map $p:\SN\longrightarrow \RP$ to note that this projection map induces a function $f': V'\longrightarrow E_1$, which satisfies \cite{KR}
\begin{eqnarray*}
f'(x')=\frac{1}{\omega_n}\int_{S'} G_{\RP_1}(x',y')dp(n(y))f'(y')d\sigma'(y'),
\end{eqnarray*}
where $x'=p(x)$, $y'=p(y)$, $S'=p(S)$ and $\sigma'$ on $S'$ is induced from $\sigma$ on $S$ by the map $p$.
Now we will assume that the domain $V$ satisfies that $-x\in V$ for each $x\in V$, the function $f$ is two-fold periodic, so that $f(x)=f(-x)$ and $S=-S$. Now the projection map $p$ gives rise to a well defined domain $V'$ on $\RP$ and a well defined function $f'(x'): V'\longrightarrow E_1$ such that $f'(x')=f(\pm x)$. As the function is spherical monogenic, i.e., $D_sf(x)=0$, we can induce a Dirac operator $D_{\RP_1}$ on $\RP$ and $D_{\RP_1}f'(x')=0$. In this case \cite{KR},
$$2f'(x')=\frac{1}{\omega_n}\int_{S'} G_{\RP_1}(x',y')dp(n(x))f'(y')d\sigma'(y').$$
Similarly, we have the conjugate of the Dirac operator $\overline{D_{\RP_1}}$, 
and the kernel of $\overline{D_{\RP_1}}$ is $\overline{G_{\RP_1}(x,y)}=\overline{G_s(x,y)}+\overline{G_s(-x,y)}$.
\par
Now we induce the Cauchy transform and its conjugate from $\SN$ to $\RP$ as follows.
\begin{eqnarray*}
T_{V'_1} f'(x')=\frac{1}{\omega_n}\int_{V'} G_{\RP_1}(x',y')f'(y')dy',\quad
\overline{T_{V'_1}} f'(x')=\frac{1}{\omega_n}\int_{V'} \overline{G_{\RP_1}(x',y')}f'(y')dy'.
\end{eqnarray*}
Also, a non-singular boundary integral operator and its conjugate are given by
\begin{eqnarray*}
F_{S'}f'(x')=\frac{1}{\omega_n}\int_{S'}G_{\RP_1}(x',y')dp(n(y'))f'(y')d\sigma'(y'),\quad
\overline{F_{S'}}f'(x')=\frac{1}{\omega_n}\int_{S'}\overline{G_{\RP_1}(x',y')}dp(n(y'))f'(y')d\sigma'(y').
\end{eqnarray*}
Hence, one obtains a Borel-Pompeiu formula as follows.
\begin{theorem}  For $f'\in C^1(V',\ClN)\cap C(\bar V')$, we have
\begin{align*}
2f'(x')=\frac{1}{\omega_n}\int_{S'}G_{\RP_1}(x',y')dp(n(y))f'(y')d\sigma'(y')
+\frac{1}{\omega_n}\int_{V'} G_{\RP_1}(x',y')D_{\RP_1}f'(y')dy'.
\end{align*}
In particular, if $f'$ has compact support, then
\begin{eqnarray*}
2f'(x')=\frac{1}{\omega_n}\int_{V'} G_{\RP_1}(x',y')D_{\RP_1}f'(y')dy',
\end{eqnarray*}
from which we can obtain $TD_{\RP_1}=2I$.
\end{theorem}
Since the domain $V=-V$, if we restrict it on the northern hemisphere, the Dirac operator $D_{\RP_1}$ is locally homeomorphic to $D_s$. Hence, we project it on the domain $V'\subset\RP$ to obtain
$$D_{\RP_1}\frac{1}{\omega_n}\int_V G_s(x,y)f(y)dy=f(x).$$
Now, for the whole domain $V$, after applying the projection on the domain $V'\subset\RP$, we have
$$D_{\RP_1}\frac{1}{\omega_n}\int_{V'} \big(G_s(x,y)+G(-x,y)\big) f'(y')dy'=2f'(x),$$
that is $D_{\RP_1}T=2I$. Similarly, we have $\overline{D_{\RP_1}T}=\overline{TD_{\RP_1}}=2I$.
\par
In the rest of this section, we will study spectrum of  the operators $\overline{D_{\RP_1}}$ and $T$, this helps us to show that the $\Pi$-operator defined in the next section also has an  $L^2$ isometry property. Similar argument can be found in \cite{CRK}.
\par
Let $H_m$ denote the space of $\ClN$-valued harmonic polynomials with homogeneity of degree $m$ on $\SN$. It is well known that $L^2(\SN)=\sum_{m=0}^\infty H_{2m}$, see \cite{Ax}. Now we consider a function $f(x)$ defined on an open domain $V\subseteq\SN$ and it also satisfies that $-x\in V$ for each $x\in V$ and $f(x)=f(-x)$. Such a domain $V$ can be projected on the real projective space $\RP$ by $p(\pm x)=x'$. Since $f(x)=\sum_{m=0}^\infty h_m(x)$, we have $f(-x)=\sum_{m=0}^\infty h_m(-x)$, and by the projection map we have $f'(x')=\sum_{m=0}^\infty h'_{2m}(x')$. Hence, $L^2(\RP)=\sum_{m=0}^\infty H'_{2m}$, where $H'_{2m}$ is the projection of $H_{2m}$ on the real projective space.
\par
Assume that $P_m$ is the space of spherical $\ClN$-valued left monogenic polynomials with homogeneity of degree $-m$ and $Q_m$ is the space of spherical $\ClN$-valued left monogenic polynomials with homogeneity of degree $n+m$, $m=0,1,2,...$. We have already known that $H_m=P_m\bigoplus Q_m$ on $\SN$ (see \cite{Balinsky}), that is for each $h_m(x)\in H_m(\SN)$ there exist $p_m(x)\in P_m(\SN)$ and $q_m(x)\in Q_m(\SN)$ such that $h_m(x)=p_m(x)+q_m(x)$. Hence $h_m(-x)=p_m(-x)+q_m(-x)$ and by the projection map, we have similar decomposition on the real projective space as $h'_{2m}(x')=p'_{2m}(x')+q'_{2m}(x')$. In other words, $L^2(\RP)=\sum_{m=0}^\infty P'_{2m}\bigoplus Q'_{2m}$. As we know that $D_s(P_m)=Q_m$ and $D_s(Q_m)=P_m$, we also have $D_{\RP_1}(P'_{2m})=Q'_{2m}$ and $D_{\RP_1}(Q'_{2m})=P'_{2m}$. Hence $D_{\RP_1}$ maps $L^2(\RP)$ to itself, similarly for $\overline{D_{\RP_1}}$. From the result in the unit sphere case, we have the spectrum of the real projective Dirac operator as follows.
\begin{align*}
\sigma(D_{\RP_1})=\sigma(\overline{D_{\RP_1}})=\{-2m-n, m=0,1,2,...\}\cup \{2m+n, m=0,1,2,...\}.
\end{align*}
Since we previously mentioned that $\overline{D_{\RP_1}T}=\overline{TD_{\RP_1}}=2I$, and $T:Q'_{2m}\longrightarrow P'_{2m}$ and $T: P'_{2m}\longrightarrow Q'_{2m}$, then the spectrum of $T$ and its conjugation $\overline{T}$ on the real projective space are $$\sigma(\overline{T})= \sigma(T)=\{\frac{2}{2m+n}, m=0,1,2,...\}\cup\{\frac{2}{-2m-n}, m=0,1,2,...\}.$$
\subsection{Construction of $\Pi$-operator on the real projective space}
We first give the definition for the $\Pi$-operator on the real projective space as follows.
\begin{definition}
The $\Pi$-operator on the real projective space is defined by $\Pi_{\RP_1}=\frac{1}{2}\overline{D_{\RP_1}}T.$
\end{definition}
The constant $\frac{1}{2}$ allows $\Pi_{\RP_1}$ to be $L^2$ isometric, we will see more details below. One can also see that $\Pi_{\RP_1}$ maps $L^2(\RP)$ to $L^2(\RP)$. 
\begin{theorem}
$\Pi_{\RP_1}$ is an $L^2(\RP)$ isometry.
\end{theorem}
\begin{proof}
We only need to prove for the function $u\in C^1(\RP)\subset L^2(\RP)$, since $C^1(\RP)$ is dense in  $L^2(\RP)$. For such a function $u$, we have the decomposition
\begin{eqnarray*}
u=\displaystyle\sum_{m=0}^\infty\sum_{p'_{2m}\in P'_{2m}}p'_{2m}+\sum_{m=0}^{-\infty}\sum_{q'_{2m}\in Q'_{2m}}q'_{2m}.
\end{eqnarray*}
Hence, with similar arguments as in \cite{CRK}, we have
\begin{align*}
&||\frac{1}{2}\overline{D_{\RP_1}}Tu||^2_{L^2}=
\displaystyle\sum_{m=0}^\infty(\frac{1}{2m+n})^2\sum_{q'_{2m}\in Q'_{2m}}\|\overline{D_{\RP_1}}q'_{2m}\|_{L^2}
+\sum_{m=0}^{\infty}(\frac{1}{-2m-n})^2\sum_{p'_m\in P'_{2m}}\|\overline{D_{\RP_1}}p'_{2m}||_{L^2}\\
=&\displaystyle\sum_{m=0}^\infty(\frac{1}{2m+n})^2(2m+n)^2\sum_{p'_{2m}\in P'_{2m}}\|p'_{2m}||_{L^2}
+\sum_{m=0}^{\infty}(\frac{1}{-2m-n})^2(-2m-n)^2\sum_{q'_{2m}\in Q'_{2m}}\|q'_{2m}||_{L^2}\\
=&\displaystyle\sum_{m=0}^\infty\sum_{p'_{2m}\in P'_{2m}}||p'_{2m}||_{L^2}+\sum_{m=0}^{\infty}\sum_{q'_{2m}\in Q'_{2m}}||q'_{2m}||_{L^2}
=||u||_{L^2}.
\end{align*}
\end{proof}
It is worth pointing out that we can assign another bundle $E_2$ to $\RP$ by identifying the pair $(x, X)$ with $(-x, -X)$, where $x\in \SN$ and $X\in \ClN$. In this case, the projection map $p$ induces a Cauchy kernel $G_{\RP_2}$, which is antiperiodic with respect to $\Gamma=\{\pm 1\}$. Hence $G_{\RP_2}(x'-y')=G_s(x,y)-G_s(-x,y)$. Further, a Clifford holomorphic function $f: V\longrightarrow \ClN$ satisfying $f(x)=-f(-x)$ will give a Clifford holomorphic function $f: V'\longrightarrow E_2$. Similarly, we can induce another Cauchy transform and its conjugate from $\SN$ to $\RP$ as follows.
\begin{eqnarray*}
T_{V'_2} f'(x')=\frac{1}{\omega_n}\int_{V'} G_{\RP_2}(x'-y')f'(y')dy',\quad
\overline{T_{V'_2}} f'(x')=\frac{1}{\omega_n}\int_{V'} \overline{G_{\RP_2}(x'-y')}f'(y')dy'.
\end{eqnarray*}
With similar arguments as for $D_{\RP_1}$, we can define $D_{\RP_2}$ on $\RP$ with the bundle $E_2$, and the $\Pi$-operator is defined as $\Pi_{\RP_2}=\frac{1}{2}D_{\RP_2}T_{V'_2}$. Similar arguments as for $\Pi_{\RP_1}$ shows that $\Pi_{\RP_2}$ also possesses the $L^2$ isometry property.
\section{$\Pi$-Operators on Cylinders and Hopf Manifolds}
Let $X$ be the cylinders $C_k$ with the measure $\eta$ by pushing forward the Lebesgue measure on $\RN$ via the quotient map $\RN\longrightarrow\RN/ \mathbb{Z}^k$. Meanwhile, $H=L^2(\C_k,\R)$ is a real Hilbert space, and $H\otimes \ClN$ is a Clifford-Hilbert module with the inner product
\begin{eqnarray*}
\langle f,g\rangle=\int_{V'}\overline{f}gd\eta(x),
\end{eqnarray*}
where $V'$ is a domain on the cylinder $C_k$ with $\overline{V'}$ enclosed and $f, g: V'\longrightarrow \ClN$. Therefore we can construct the $\Pi$-operator theory on cylinders as demonstrated in the previous section.
\par
Similarly, if we let $X$ to be Hopf manifolds $\mathbb{S}^1\times \SN$ with the pushforward measure obtained via the quotient map defined below, and $H=L^2(\mathbb{S}^1\times \SN,\R)$ is a Hilbert space. Then we can build the $\Pi$-operator theory on the Clifford-Hilbert module $H\otimes \ClN$. More details are given below.
\subsection{$\Pi$-Operators on Cylinders}
For integer $k$, $1\leq k\leq n$, we define the $k$-cylinder $C_k$ to be the manifold $\RN/\mathbb{Z}^k$ where $\mathbb{Z}^k=\mathbb{Z}e_0+\mathbb{Z}e_1+...+\mathbb{Z}e_{k-1}$. Each element in $C_k$ has the form $m_0e_0+\cdots m_{k-1}e_{k-1}$ for $m_0,\cdots,m_{k-1}\in\mathbb{Z}$ and it is denoted by $\underline{t}$. For each $k$, the space $\RN$ is the universal covering space of the cylinder $C_k$. Hence, there is a projection map $p_k: \RN\longrightarrow C_k$.
\par
Let $U$ be an open subset of $\RN$. It is called \emph{$k$-fold periodic} if for each $x\in U$ and $t\in\mathbb{Z}^k$ we also have $x+\underline{t}\in U$. Hence, $U'=p_k(U)$ is an open subset of $C_k$. Suppose that $U\subseteq\RN$ is a $k$-periodic open set and $f(x)$ is a $\mathcal{C}l_n$-valued function defined on $U$, we say that $f(x)$ is a \emph{$k$-fold periodic function} if we have $f(x)=f(x+\underline{t})$ for each $x\in U$. Hence, the projection $p_k$ induces a well defined function $f':\ U'\longrightarrow \ClN$, where $f'(x')=f(x)$ for each $x'\in U'$ and $x$ is an arbitrary representative of $p_k^{-1}(x')$. Moreover, any function $f':\ U'\longrightarrow \ClN$ can lift to a $k$-fold periodic function $f:\ U\longrightarrow \ClN$, where $U=p_k^{-1}(U')$.
\par
In \cite{KR1}, the spinor bundle over $C_k$ is trivial on $C_k\times\ClN$. Other $k$ spinor bundles $E^{(l)}$ over $C_k$ arise by making the identification $(x,X)$ with $(x+\underline{m}+\underline{n}, (-1)^{m_0+m_1+...+m_l}X)$, where $l$ is an integer and $0\leq l\leq k$, $\underline{m}$ is in the lattice $\mathbb{Z}^l=\mathbb{Z}e_0+\mathbb{Z}e_1+...+\mathbb{Z}e_{l-1}$, and $\underline{n}$ is in the lattice $\mathbb{Z}^{k-l}=\mathbb{Z}e_l+\mathbb{Z}e_{l+1}+...+\mathbb{Z}e_{k-1}$.
\par
Let $G(x,y)=\frac{\overline{x-y}}{||x-y||^{n+1}}$ be the fundamental solution of the Euclidean Dirac operator. Consider the series
$
\cot_{k,0}(x,y)=\sum_{\underline{m}\in \mathbb{Z}^k}G(x-y+\underline{m})
$
which converges on $\RN\setminus \mathbb{Z}^k$, for $k< n-1$, see \cite{KR}. Then, the kernel of the Dirac operator on the cylinder $C_k$ with the trivial bundle has the form $\cot_{k,0}(x',y')$ which is defined on $(C_k\times C_k)\setminus diagonal(C_k)$, where $diagonal(C_k)=\{(x',x'): x'\in C_k\}$.
More generally, For $k<n-1$ and $l\leq k$, the kernel $\cot_{k,l}(x',y')$ of the Dirac operator on $C_k$ with the bundle $E^{(l)}$ is given by applying $p_k$ to
\begin{eqnarray*}
\cot_{k,l}(x,y)=\sum_{\underline{m}\in \mathbb{Z}^k,\underline{n}\in \mathbb{Z}^{l}}(-1)^{m_0+m_1+...m_{l-1}}G(x-y+\underline{m}+\underline{n}).
\end{eqnarray*}
Further, with the projection map $p_k$, we can otain the Dirac operator on $C_k$ with the bundle $E^{(l)}$, which is denoted by $D_l$. Similar argument applies for the conjugation $\overline{D_l}$ and its fundamental solution $\overline{\cot_{k,l}(x',y')}$. One also has $D_l\overline{D_l}=\overline{D_l}D_l=\Delta_l$, where $\Delta_l$ is a spinorial Laplacian, see \cite{KR}.
\par
Suppose $f:V\longrightarrow\RN$ satisfying $f(x+\underline{m}+\underline{n})=(-1)^{m_0+m_1+...m_{l-1}}f(x)$, where $\underline{m}\in \mathbb{Z}^k,\underline{n}\in \mathbb{Z}^{l}$. Then, $f$ can be lifted  by the projection map $p_k$ to a function $f':V'\longrightarrow E^{(l)}$, where $V'=p_k^{-1}(V)$. If $D_l f'=0$, $f'$ is called an $E^{(l)}$ left Clifford monogenic function.
\par
Using the fundamental solutions of the Dirac operators, we can define the Cauchy transform on different bundles. If $f':V'\longrightarrow E^{(l)}$, $S'$ is a surface lying in $V'$ and bounding a subdomain $W'$. Suppose $x'\in W'$, then
\begin{eqnarray*}
T_{V'} f'(x')=\frac{1}{\omega_{n}}\int_{V'} \cot_{k,l}(x',y')f'(y')dy',\quad
\overline{T_{V'}} f'(x')=\frac{1}{\omega_{n}}\int_{V'} \overline{\cot_{k,l}(x',y')}f'(y')dy'.
\end{eqnarray*}
Also, a non-singular boundary integral operator and its conjugate are given by
\begin{eqnarray*}
F_{S'}f'(x')=\frac{1}{\omega_{n}}\int_{S'}\cot_{k,l}(x',y')dp(n(y'))f'(y')d\sigma'(y'),\quad
\overline{F_{S'}}f'(x')=\frac{1}{\omega_{n}}\int_{S'}\overline{\cot_{k,l}(x',y')}dp(n(y'))f'(y')d\sigma'(y').
\end{eqnarray*}
Hence, the Borel-Pompeiu formula in this context is stated as follows.
\begin{theorem} \cite{KR1}  For $f'\in C^1(V',\ClN)\cap C(\overline{ V'})$, we have
\begin{align*}
f'(x')=&\frac{1}{\omega_{n}}\big(\int_{S'}\cot_{k,l}(x',y')dp(n(y))f'(y')d\sigma'(y')
+\frac{1}{\omega_{n}}\int_{V'} \cot_{k,l}(x',y')D_lf'(y')dy'\big).
\end{align*}
\end{theorem}
Similar as the case in Euclidean space, for a function $f'$ with compact support, we have $D_lT_{V'}=T_{V'}D_l=I$, and $\overline{D_lT_{V'}}=\overline{T_{V'}D_l}=I$ as well.
\par
Now we  define the $\Pi$-operator on the cylinder as follows.
\begin{definition}
The $\Pi$-operator on the cylinder is defined by $\Pi_l=\overline{D_l}T.$
\end{definition}
Since $\Pi_l$ is induced from the $\Pi$-operator in Euclidean space, we expect similar results as in \cite{GK}.
\begin{theorem}
$\Pi_l$ is an $L^2(C_k)$ isometric operator.
\end{theorem}
\begin{proof}
The proof is similar to the proof of Proposition 5 in \cite{GK}.
\end{proof}
In this section, we will use the norm estimation of the $\Pi$-operator on the cylinder to determine existence of the solution of Beltrami equation on the cylinder. First, we define the Beltrami equation on the cylinder as follows.
\par
Let $V' \subseteq C_k$ be a bounded, simply connected domain with sufficiently smooth boundary, and $q, f': V'\longrightarrow E^{(l)}$,  q is a measurable function, and $f'$ is sufficiently smooth. The Beltrami equation on the cylinder is as follows.
$$ D_lf'=q\overline{D_l}f'.$$
We already explained how the estimate of the norm of $\Pi$-operator determines the existence and uniqueness of solutions to Beltrami equations in Section $3$. Next, we will provide a norm estimate for our $\Pi$-operator here.
\par
Suppose $V=\bigcup_{i=1}^\infty V_i=p_k^{-1}(V')$, such that $p_k(V_i)=V',\ i=1,2,\cdots$. $f:V_i\longrightarrow\ClN$ is a piecewise continuous function with compact support, and $f$ can be induced to $f':V'\longrightarrow E$. For the $\Pi$-operator on $\RN$, we have $\|\Pi\|_{L^p(\RN)}\leq (n+1)(p^*-1)$, where $p^*=max(p, p/(p-1))$, see \cite{NW}.
\par
Recall that $\Pi_l=\overline{D_l}\cot_{q,k,0}\ast$, where $``\ast"$ is the standard convolution. On each subdomain $V_i$, we have $\|\overline{D_l}\cot_{q,k,0}\ast f'(x')\|_{L^p(V_i)}=\|\overline{D}G\ast f(x')\|_{L^p(V_i)}\leq (n+1)(p^*-1)\|f(x')\|_{L^p(V_i)}$. Hence for the domain $V=\bigcup_{i=1}^\infty V_i$, we have $\|\overline{D_l}\cot_{q,k,0}\ast f(x)\|_{L^p(V)}=\|\overline{D}G\ast f(x)\|_{L^p(V)}\leq (n+1)(p^*-1)\|f(x)\|_{L^p(V)}$. Applying the projection $p_k$ on $V$, we could obtain
\begin{theorem}
 $$\|\overline{D'}\cot'_{q,k,0}\ast f(x')\|_{L^p(V')}\leq (n+1)(p^*-1)\|f(x')\|_{L^p(V')}, $$
 \end{theorem}
 which shows $\|\Pi_l\|_{L^p(C_k)}\leq (n+1)(p^*-1)$, where $p^*=max(p, p/(p-1))$.
\subsection{$\Pi$-operators on Hopf Manifolds}
A Hopf manifold is diffeomorphic to the conformally flat spin manifold $U/\Gamma=\mathbb{S}^1\times \SN$, where $U=\RN\setminus\{0\}$ and $\Gamma=\{2^k:k\in Z\}$. There exists a projection $p_k:\RN\setminus\{0\}\longrightarrow \mathbb{S}^1\times \SN$, such that $p_k(2^kx)=x'$.
\par
Let $V\subseteq \RN$ be open, and if $x\in V$, $2^kx\in V$. Hence $p_k(V)=V'\subseteq \mathbb{S}^1\times \SN$, which is also open. A left Clifford holomorphic functions $f: V\longrightarrow \ClN$ which satisfying $f(x)=f(2^kx)$ can be lifted to a well defined function $f':V'\longrightarrow \ClN$ by the projective map $p_k$, where $f'(x')=f(x)$ for each $x'\in V'$ and $x$ is one of $p_k^{-1}(x')$.
\par
The spinor bundle $E$ over $\mathbb{S}^1\times \SN$ is constructed by identifying $(x, X)$ with $(2^kx, X)$ for $k\in Z$ and $x\in \RN\setminus\{0\}$, $X\in \ClN$. By \cite{KR1}, the Cauchy kernel for $\mathbb{S}^1\times \SN$ is given as follows. Let $C(x-y)=C_1(x-y)+2^{2-2n}C_2(x-y)$, where
\begin{eqnarray*}
C_1(x-y)=\sum_{k=0}^\infty G(2^kx-2^ky),\quad
C_2(x-y)=G(x)\sum_{k=-1}^{-\infty} G(2^{-k}x^{-1}-2^{-k}y^{-1})G(y).
\end{eqnarray*}
$G(x,y)=\frac{\overline{x-y}}{||x-y||^n}$ is the fundamental solution of the Euclidean Dirac operator. After applying the projective map, we obtain the Cauchy kernel $C'(x'-y')$ for the Dirac operator on $(\mathbb{S}^1\times \SN)\times(\mathbb{S}^1\times \SN)\setminus diagonal(\mathbb{S}^1\times \SN)$, which is denoted as $D'$.  A function $f'$ defined on $V'\subseteq \mathbb{S}^1\times \SN$ is left monogenic if $D'f'=0$.
\par
Using the kernel of the Dirac operators $D'$, we can define the Cauchy transform on $S^1\times \SN$. If $f':V'\longrightarrow E$, $S'$ is a surface lying in $V'$ and bounding a subdomain $W'$. Suppose $x'\in W'$,
\begin{eqnarray*}
T_{V'} f'(x')=\frac{1}{\omega_{n+1}}\int_{V'} C(x'-y')f'(y')dy',\quad
\overline{T_{V'}} f'(x')=\frac{1}{\omega_{n+1}}\int_{V'} \overline{C(x'-y')}f'(y')dy'.
\end{eqnarray*}
Also, a non-singular boundary integral operator and its conjugate are given by
\begin{eqnarray*}
F_{S'}f'(x')=\frac{1}{\omega_{n+1}}\int_{S'}C(x'-y')dp(n(y'))f'(y')d\sigma'(y'),\quad
\overline{F_{S'}}f'(x')=\frac{1}{\omega_{n+1}}\int_{S'}\overline{C(x'-y')}dp(n(y'))f'(y')d\sigma'(y').
\end{eqnarray*}
And the Borel-Pompeiu formula is stated as follows.
\begin{theorem} \cite{KR1} For $f'\in C^1(V',\ClN)\cap C(\overline {V'})$, we have
\begin{eqnarray*}
f'(x')=\frac{1}{\omega_{n+1}}\big(\int_{S'}C(x'-y')dp(n(y))f'(y')d\sigma'(y')+\int_{V'} C(x'-y')D_lf'(y')dy'\big).
\end{eqnarray*}
\end{theorem}
\begin{definition}
Define the $\Pi$-operator on the Hopf manifold as $$\Pi'f'=\overline{D'}Tf'.$$
\end{definition}
Since $\Pi'$ is induced from the $\Pi$-operator in Euclidean space, we expect similar results as in \cite{GK}.
\begin{theorem}
$\Pi'$ is an $L^2$ isometry operator.
\end{theorem}
\begin{proof}
The proof is similar as for the Proposition 5 in \cite{GK}.
\end{proof}
Now, we will introduce a norm estimate for the $\Pi$-operator in this context. Let $V' \subseteq \mathbb{S}^1\times \SN$ be a bounded, simply connected domain with sufficiently smooth boundary, and $q, f': V'\longrightarrow E$,  q is a measurable function, and $f'$ is sufficiently smooth. The Beltrami equation on the Hopf manifold is as follows.
$$ D'f'=q\overline{D'}f'.$$ 
Suppose $V=\bigcup_{i=1}^\infty V_i$ is the inverse image of $V'$ under $p_k$, such that $p_k(V_i)=V'$. $f:V_i\longrightarrow\ClN$ is a piecewise continuous function with compact support, and $f$ could be induced to $f':V'\longrightarrow E$. For the $\Pi$-operator on $\RN$, we have $\|\Pi\|_{L^p(\RN)}\leq (n+1)(p^*-1)$, where $p^*=max(p, p/(p-1))$, see \cite{NW}.
\par
On each subdomain $V_i$, we have $\|\overline{D'}C*f(x)\|_{L^p(V_i)}=\|\overline{D}G*f(x)\|_{L^p(V_i)}$, hence for the domain $V=\sum_{i=1}^\infty V_i$, we have $\|\overline{D'}C*f(x)\|_{L^p(V)}=\|\overline{D}G*f(x)\|_{L^p(V)}\leq (n+1)(p^*-1)\|f(x)\|_{L^p(V)}$. Applying the projection $p_k$ on $V$, we can obtain $\|\overline{D'}C'*f(x')\|_{L^p(V')}\leq (n+1)(p^*-1)\|f(x')\|_{L^p(V¡®)}$, which shows the following.
\begin{theorem}
$\|\Pi'\|_{L^p(\mathbb{S}^1\times \SN)}\leq (n+1)(p^*-1)$, where $p^*=max(p, p/(p-1))$.
\end{theorem}
\section{A $\Pi$-Operator on the Hyperbolic upper-half Space}\label{hyperbolic}
Let $X$ be the upper-half space $\Ru$ with the hyperbolic measure. Then the Hilbert space $H=L^2(\Ru,\R)$ becomes a real Hilbert space, and $H\otimes \ClN$ is a Clifford-Hilbert module with the inner product
\begin{eqnarray*}
\langle f,g\rangle=\int_{\Omega}\overline{f}g\frac{dx^n}{x_n^{n-1}},
\end{eqnarray*}
where $\Omega$ is a subset of the upper-half space with $\overline{\Omega}$ inclosed and $f, g: \Omega\longrightarrow \ClN$. Then the $\Pi$-operator theory on the hyperbolic upper-half space is actually a special case of Section $3$, which is demonstrated as follows.
\subsection{Hyperbolic Dirac Operator}
Denote the upper-half space $\Ru=\{x_0e_0+x_1e_1\cdots+x_ne_n : x_n>0\}$. The Poincar\'{e} half-space is a Riemannian manifold $(\Ru, ds^2)$ with the Riemannian metric
$ds^2=\displaystyle \frac{(dx_0^2+dx_1^2+....+dx_n^2)}{x_n^2}.$
The Clifford algebra $\ClN$ can be expressed as $\ClN=\Cln+\Cln e_n$. So if $A\in \ClN$, there exist unique elements $B$ and $C\in \Cln$, such that $A=B+Ce_n$. This gives rise to a pair of projection maps $P$ and $Q$, where
$$P:\ClN\longrightarrow \Cln, P(A)=B,\quad Q:\ClN\longrightarrow \Cln, Q(A)=C.$$
We denote $-e_nQ(A)e_n$ by $Q'(A)\in \Cln$. The modified Dirac operator is defined as
$$Mf=D_0f+\displaystyle\frac{n-1}{x_n}Q'f,$$
where $D_0=\sum_{i=0}^ne_i\partial_{x_i}$ is the Dirac operator on $\RN$. Let $\Omega\subset\Ru$, we say a function $f:\Omega\longrightarrow\ClN$ is \emph{hypermonogenic} if $Mf(x)=0$ for each $x\in \Omega$.
\par
The conjugate of the modified Dirac operator is defined by
$\overline{M}f=\overline{D_0}f-\frac{n-1}{x_n}Q'f,$
where $\overline{D_0}=e_0\partial_{x_0}-\sum_{i=i}^ne_i\partial_{x_i}$, see \cite{Qiao}. Next result shows the relation between $M$ and $\overline{M}$.
\begin{proposition}
$M^*=-\overline{M}$.
\end{proposition}
\begin{proof}
Let $f,g \in L^2(\Ru,\ClN)$ with compact support. From the decomposition that $A=P(A)+Q(A)e_n$, we notice that $||f||^2_h=||Pf||^2_h+||Qf||^2_h$, where
$$||f||^2_h=\int_\Omega \overline{f(x)}f(x)\frac{dx^n}{x_n^{n-1}}$$
  defines the norm of $f$ in the upper-half space with hyperbolic metric. If we replace $f$ that in the previous identity with $f+g$, one can easily see that $P(f)$ is orthogonal to $Q(g)e_n$. More specifically,
 \begin{eqnarray}\label{ortho}
 \int_\Omega \overline{P(f)}\cdot (Q(g)e_n)\frac{dx^n}{x_n^{n-1}}=0.
 \end{eqnarray}
On one hand, since we have
$$\langle Mf,g\rangle=\langle\sum_{i=0}^ne_i\frac{\partial f}{\partial x_i}+\frac{n-1}{x_n}Q'f,g\rangle=\langle\sum_{i=0}^ne_i\frac{\partial f}{\partial x_i}-\frac{n-1}{x_n}e_nQfe_n,g\rangle,$$
then
\begin{align*}
&\langle\sum_{i=0}^ne_i\frac{\partial f}{\partial x_i},g\rangle=\int_\Omega\overline{\sum_{i=0}^ne_i\frac{\partial f}{\partial x_i}}\cdot g\frac{dx^n}{x_n^{n-1}}
=\int_\Omega\overline{\sum_{i=0}^n\frac{\partial f}{\partial x_i}}\cdot\overline{e_i}g\frac{dx^n}{x_n^{n-1}}
=-\int_\Omega \overline{f}\cdot\sum_{i=0}^n\frac{\partial}{\partial x_i}(\overline{e_i}g)\frac{dx^n}{x_n^{n-1}}\\
=&-\int_\Omega \overline{f} \big(\sum_{i=0}^n \overline{e_i}\frac{\partial g}{\partial x_i}\frac{dx^n}{x_n^{n-1}}\big)-\int_\Omega \overline{f} \overline{e_n}g\frac{-(n-1)}{x_n^n}dx^n   
=\langle f, -\overline{D_0}g\rangle-(n-1)\int_\Omega \overline{f}\cdot e_ng\frac{dx^n}{x_n^n}.
\end{align*}
On the other hand,
\begin{eqnarray*}
\langle-\frac{n-1}{x_n}e_nQfe_n,g\rangle=-(n-1)\int_\Omega\overline{e_nQfe_n}g\frac{dx^n}{x_n^{n}}
=(n-1)\int_\Omega \overline{Qfe_n}\cdot e_n g\frac{dx^n}{x_n^{n}}.
\end{eqnarray*}
Hence,
\begin{align*}
&\langle Mf,g\rangle=\langle\sum_{i=0}^ne_i\frac{\partial f}{\partial x_i}-\frac{n-1}{x_n}e_nQfe_n,g\rangle
=\langle f, -\overline{D_0}g\rangle-(n-1)\int_\Omega \overline{f}\cdot e_ng\frac{dx^n}{x_n^n}+(n-1)\int_\Omega \overline{Qfe_n}\cdot e_n g\frac{dx^n}{x_n^{n}}\\
=&\langle f, -\overline{D_0}g\rangle-(n-1)\int_\Omega \overline{Pf}\cdot e_ng\frac{dx^n}{x_n^n}
=\langle f, -\overline{D_0}g\rangle-(n-1)\int_\Omega \overline{Pf}\cdot e_n(Pg+Qge_n)\frac{dx^n}{x_n^n}.
\end{align*}
Since $e_nPg$ can be rewritten as $\pm Pge_n$, where $``\pm"$ depends on whether $n$ is even or odd. This can also be considered as $Qhe_n$ for some function $h\in L^2(\Ru,\ClN)$. Hence, from (\ref{ortho}), we can see that $Pf$ is orthogonal to $e_nPg$. Thus, the previous equation becomes
\begin{eqnarray*}
\langle f, -\overline{D_0}g\rangle-(n-1)\int_\Omega \overline{Pf}\cdot e_nQge_n\frac{dx^n}{x_n^n}.
\end{eqnarray*}
With a similar argument as above, the previous equation is equal to
\begin{align*}
=&\langle f, -\overline{D_0}g\rangle-(n-1)\int_\Omega \overline{Pf+Qfe_n}\cdot e_nQge_n\frac{dx^n}{x_n^n}
=\langle f, -\overline{D_0}g\rangle-(n-1)\int_\Omega \overline{f}\cdot e_nQge_n\frac{dx^n}{x_n^n}\\
=&\langle f,-\overline{D_0}g+\frac{n-1}{x_n}Q'g\rangle=\langle f,-\overline{M}g\rangle.
\end{align*}
Therefore, $M^*=-\overline{D_0}+\frac{n-1}{x_n}Q'=-\overline{M}$. Similarly, $\overline{M}^*=-M$.
\end{proof}
By a straight forward calculation, we can obtain
\begin{eqnarray*}
M\overline{M}f=\overline{M}Mf=\Delta f-\frac{n-1}{x_n}\frac{\partial}{\partial_{x_n}}f+(n-1)\frac{Qfe_n}{x_n^2},
\end{eqnarray*}
where $\Delta$ is the Laplace operator in $\RN$.
In the hyperbolic function theory, we define hyperbolic harmonic function $f:\Omega\longrightarrow\ClN$ as a solution of the equation
\begin{align*}
\overline{M}Mf(x)=0,\quad \text{for}\ x\in \Omega.
\end{align*}
Let
\begin{eqnarray*}
E(x,y)=\displaystyle \frac{(x-y)^{-1}}{\|x-y\|^{n-1}\|x-\widehat{y}\|^{n-1}},\
F(x,y)=\displaystyle \frac{(\widehat{x}-y)^{-1}}{\|x-y\|^{n-1}\|\widehat{x}-y\|^{n-1}},
\end{eqnarray*}
where $\widehat{x}=\sum_{i=0}^{n-1}x_ie_i-x_ne_n$. The Cauchy transform is defined as \cite{EO}
\begin{eqnarray*}
T_\Omega f(y)=-\displaystyle \frac{2^{n-1}y_n^{n-1}}{\omega_{n+1}}\int_\Omega \big(E(x,y)f(x)-F(x,y)\widehat{f(x)}\big)dx^n,
\end{eqnarray*}
where $\widehat{f}=\sum_{i=0}^{n-1}f_ie_i-f_ne_n$. Also, a non-singular boundary integral operator is given by
\begin{eqnarray*}
F_{\partial \Omega}f(y)=\displaystyle \frac{2^{n-1}y_n^{n-1}}{\omega_{n+1}}\int_{\partial \Omega}\big(E(x,y)n(x)f(x)-F(x,y)\widehat{n}(x)\widehat{f}(x)\big)d\sigma(x).
\end{eqnarray*}
Hence, we have a Borel-Pompeiu formula as follows.
\begin{theorem}\cite{EO}
Let $\Omega\subseteq \Ru$ be a bounded region with smooth boundary in $\Ru$. Suppose $f:\Omega\longrightarrow \ClN$ is a $C^1$ function on $\Omega$ with a continuous extension to the closure of $\Omega$. Then for $y\in \Omega$, we have
\begin{eqnarray*}
f(y)=\displaystyle \frac{2^{n-1}y_n^{n-1}}{\omega_{n+1}}\int_{\partial \Omega}\big(E(x,y)n(x)f(x)-F(x,y)\widehat{n}(x)\widehat{f}(x)\big)d\sigma(x)
-\displaystyle \frac{2^{n-1}y_n^{n-1}}{\omega_{n+1}}\int_\Omega \big(E(x,y)Mf(x)-F(x,y)\widehat{Mf(x)}\big)dx^n.
\end{eqnarray*}
\end{theorem}
\begin{remark}
We notice that when $f$ is a hypermonogenic function, we have
\begin{eqnarray*}
f(y)=\displaystyle \frac{2^{n-1}y_n^{n-1}}{\omega_{n+1}}\int_{\partial \Omega}\big(E(x,y)n(x)f(x)-F(x,y)\widehat{n}(x)\widehat{f}(x)\big)d\sigma(x).
\end{eqnarray*}
Further, if $f\in \wzwo(\Omega,{\mathcal{C}l_n})$, then
\begin{eqnarray*}
f(y)=-\displaystyle \frac{2^{n-1}y_n^{n-1}}{\omega_{n+1}}\int_\Omega \big(E(x,y)Mf(x)-F(x,y)\widehat{Mf(x)}\big)dx^n,
\end{eqnarray*}
in other words, $TM=I$. If we apply the hyperbolic Dirac operator $M$ on both sides of the equation, we can easily obtain $MT=I$.
\end{remark}
\subsection{Construction of the Hyperbolic $\Pi$-Operator}
The generalization of $\Pi$-operator to higher dimension via Clifford algebras is defined as follows.
\begin{definition}
The hyperbolic $\Pi$-operator in $\Ru$ is defined as
$\Pi_h=\overline{M}T.$
\end{definition}
The following are some well known properties for the $\Pi_h$-operator.
\begin{theorem}
Suppose $f\in \wzwop(\Omega)\ (1<p<\infty, k\geq 1)$, then
\begin{enumerate}
\item $M\Pi_h f=\overline{M}f$, $\Pi_h Mf=\overline{M}f-\overline{M}F_{\partial\Omega}f,$
\item $F_{\partial\Omega}\Pi_h f=(\Pi_h-T\overline{M})f$, $M\Pi_h f-\Pi_h Mf=\overline{M}F_{\partial\Omega}f.$
\end{enumerate}
\end{theorem}
\begin{proof}
The proof is a straight forward calculation.
\end{proof}
The following decomposition of $L^2(\Omega,\ClN)$ helps us to observe that the $\Pi$-operator actually maps $L^2(\Omega,\ClN)$ to $L^2(\Omega,\ClN)$.
\begin{theorem}  \text{(\textbf{Decomposition of $L^2(\Omega,\ClN$)})}
$$L^2(\Omega,\ClN)=L^2(\Omega,\ClN)\ \cap\ Ker \overline{M}\oplus M(\wzwo(\Omega, \ClN)),$$
and
$$L^2(\Omega,\ClN)=L^2(\Omega,\ClN)\ \cap\ Ker M\oplus\overline{M}(\wzwo(\Omega, \ClN)).$$
\end{theorem}
\begin{remark}
The proof is similar to the proof in \cite[Theorem 1]{GK}.
Notice that
\begin{eqnarray*}
\Pi_h(L^2(\Omega,\ClN)\cap Ker \overline{M})=L^2(\Omega,\ClN)\cap Ker M,\quad
\Pi_h(M(\wzwo(\Omega, \ClN))=\overline{M}(\wzwo(\Omega, \ClN)).
\end{eqnarray*}
Hence, $\Pi_h$ maps $L^2(\Omega,\ClN)$ to $L^2(\Omega,\ClN)$.
\end{remark}
Further,  the $\Pi$-operator is an $L^2$ isometry, which can also be demonstrated as follows.
\begin{theorem}
For functions in $L^2(\Omega,\ClN)$, we have
$\Pi^* \Pi=I.$
\begin{proof}
Let $f\in L^2(\Omega, \ClN)$ with compact support,
\begin{eqnarray*}
\langle\Pi_h f, \Pi_h f\rangle=\langle\overline{M}Tf, \overline{M}Tf\rangle=-\langle Tf, M\overline{M}Tf\rangle
=-\langle Tf, \overline{M}MTf\rangle=\langle MTf, MTf\rangle=\langle f,f\rangle.
\end{eqnarray*}
Here we use $\overline{M}^*=-M$.
\end{proof}
\end{theorem}
\begin{remark}
The argument for the applications of $\Pi$-operators on Beltrami equations  given in Section $3$ can also be applied here, which shows the further applications of the $\Pi$-operators in the study of the Beltrami equations in the hyperbolic case.
\end{remark}



\begin{thebibliography}{99}


\bibitem{Ax} S. Axler, P. Bourdon, W. Ramey, \emph{Harmonic function theory}, second edition, Graduate Texts in Mathematics, Springer, New York, (2001).

\bibitem{Balinsky} A. Balinsky, J. Ryan, \emph{Some sharp $L^2$ inequalities for Dirac type operators}, SIGMA, Symmetry Integrability Geom. Methods Appl.,  10, (2007), paper 114, electronic only.

\bibitem{Blaya} R. A. Blaya, J. B. Reyes, A. G. Ad\'{a}n, U. K\"{a}hler, \emph{On the $\Pi$-operator in Clifford analysis}, Journal of Math. Anal. Appl.,  Vol. 434, Issue 2, (2016),1138-1159.

\bibitem{Br} F. Brackx, R. Delanghe, F. Sommen, \emph{Clifford Analysis}, Pitman, London, (1982).

\bibitem{CRK} W. Cheng, J. Ryan, U. K\"{a}hler, \emph{Spherical $\Pi$-operators in Clifford Analysis and Applications}, Complex Analysis and Operator Theory,  Vol. 11, Issue 5, (2017),1095-1112.



\bibitem{EO} S.L. Eriksson, H. Orelma, \emph{A Mean-Value Theorem for Some Eigenfunctions of the Laplace-Beltrami Operator on the Upper-Half Space}, Annales Academiae Scientiarum Fennicae Mathematica,  Vol. 36, (2011),101-110.

\bibitem{GK} K. G\"{u}rlebeck,U. K\"{a}hler, \emph{On a spatial generalization of the complex $\Pi$-operator},  Zeitschrift f\"ur Analysis und ihre Anwendungen, Vol. 15, No. 2, (1996), 283-297.

\bibitem{GKS} K. G\"{u}rlebeck,U. K\"{a}hler, M. Shapiro, \emph{On the $\Pi$-operator in hyperholomorphic function theory}, Advances in Applied Clifford Algebras,  Volume 9, Issue 1, June (1999),23-40.

\bibitem{Kahler} U. K\"{a}hler, \emph{On quaternionic Beltrami equations}, Clifford Algebras and their Applications in Mathematical Physics,  Vol. 19 of the seires Progress in Physics, (2000),3-15.

\bibitem{KR} R.S. Krau{\ss}har, J. Ryan, \emph{Some Conformally Flat Spin Manifolds, Dirac Operators and Automorphic Forms}, Journal of Mathematical Analysis and Applications,  Vol. 325, Issue 1, (2007),359-376.

\bibitem{KR1} R.S.Krau{\ss}har, J. Ryan, \emph{Clifford and Harmonic Analysis on Cylinders and Tori}, Rev. Mat. Iberoamericana,  Vol. 21, No. 1, (2005),87-110.


\bibitem{LR}H. Liu, J. Ryan, \emph{Clifford analysis techniques for spherical PDE}, The Journal of Fourier Analysis and Applications,  Vol. 8, Issue 6, (2002),535-564.



\bibitem{Martin} M. Martin, \emph{Higher-Dimensional Ahlfors-Beurling Type Inequalities in Clifford Analysis}, Proceedings of the American Mathematical Society, Vol. 126, No. 10, (1998), 2863-2871.

\bibitem{NW} C.A. Nolder, G. Wang, \emph{Fourier Multipliers and Dirac Operators}, Advances in Applied Clifford Algebras,  Vol. 27, Issue 22, (2017),1647-1657.

\bibitem{Qiao} Y. Qiao, S. Bernstein, S.L. Eriksson, J. Ryan, \emph{Function Theory For Laplace And Dirac-Hodge Operator in Hyperbolic Space}, Journal d'Analyse Math\'ematique,  Volume 98, Issue 1, (2006),43-64.
\end{thebibliography}
\end{document}